\documentclass{amsart}

\usepackage{amssymb} \usepackage{amsfonts} \usepackage{amsmath}
\usepackage{amsthm} \usepackage{epsfig} 
\usepackage{color}
\usepackage{pinlabel}
\usepackage{amscd}
\usepackage[all]{xy}
\usepackage{pinlabel}
\usepackage{tcolorbox}
\usepackage{seqsplit}
\usepackage{accents}

\usepackage{tikz}
\usepackage{xcolor}

\newtheorem{lemma}{Lemma}
\newtheorem{teo}[lemma]{Theorem}

\newtheorem{cor}[lemma]{Corollary} 
\newtheorem{conj}[lemma]{Conjecture}

\theoremstyle{definition}
\newtheorem{defn}[lemma]{Definition}

\newtheorem{example}[lemma]{Example}

\theoremstyle{remark}

\newcommand{\matr} [4] {\big({\tiny\begin{array}{@{}c@{\ }c@{}} #1 & #2 \\ #3 & #4 \\ \end{array}} \big)}

\newcommand{\interior}[1]{{\rm int}(#1)}

\newcommand{\PSL}{{\rm PSL}}
\newcommand{\SL}{{\rm SL}}
\newcommand{\Nil}{{\rm Nil}}
\newcommand{\Sol}{{\rm Sol}}
\newcommand{\tr}{{\rm tr}}
\newcommand{\isom}{\cong}

\newcommand{\matX}{\ensuremath {\mathbb{X}}}

\newcommand{\matR} {\ensuremath {\mathbb{R}}}
\newcommand{\matQ} {\ensuremath {\mathbb{Q}}}
\newcommand{\matZ} {\ensuremath {\mathbb{Z}}}
\newcommand{\matC} {\ensuremath {\mathbb{C}}}

\newcommand{\matH} {\ensuremath {\mathbb{H}}}

\newcommand{\matCP} {\ensuremath {\mathbb{CP}}}

\newcommand{\SO} {\ensuremath {{\rm SO}}}
\newcommand{\id} {\ensuremath {{\rm id}}}

\newcommand{\Isom} {\ensuremath {{\rm Isom}}}

\author{Bruno Martelli}

\title{Geometrisation of 3-manifolds}

\begin{document}

\begin{abstract}
The geometrisation theorem of 3-manifolds was conjectured by Thurston the 1980s and proved by Perelman in the 2000s. This is an overview on the subject. We explain the content of the theorem and describe its effects in various situations.
\end{abstract}

\maketitle

\section*{Introduction}
The geometrisation theorem of 3-manifolds was conjectured by Thurston in a famous paper \cite{Thu} published by the Bulletin of the AMS in 1982, and then proved by Perelman in a sequence of three papers \cite{Per1, Per3, Per2} posted on the arXiv in 2002 and 2003. The theorem may be stated essentially in two versions. The most complete one is

\begin{center}
\emph{Every 3-manifold decomposes canonically into geometric pieces.} 
\end{center}

A more focused version says

\begin{center}
\emph{Every 3-manifold is hyperbolic unless there are some obvious obstructions.} 
\end{center}

Each sentence contains a few terms that need to be carefully defined. The aim of this paper is to explain the content of both versions, and to describe geometrisation in action in various situations. We will not say a single word on Perelmann's proof, that is admittedly far outside of the author's expertise. 

{\bf Acknowlegments.} Some figures are taken from Wikipedia Commons: Figures \ref{regular:fig} and \ref{rhombic:fig} were made by Cyp and Figure \ref{satellite:fig} was made by RyBu.

\section{The theorem}

The most complete version of the geometrisation theorem for 3-manifolds says

\begin{center}
\emph{Every 3-manifold decomposes canonically into geometric pieces.} 
\end{center}

We explain the terminology. Every 3-manifold here is tacitly assumed to be compact, connected, orientable, smooth, and possibly with boundary consisting of a disjoint union of tori. A \emph{decomposition} is a two steps operation, where we first \emph{cut} the manifold along a canonical family of \emph{spheres}, and then we cut it again along a canonical family of \emph{tori}, see Figure \ref{surfaces:fig}. We will explain these two steps carefully.
A \emph{geometric piece} is a 3-manifold whose interior may be equipped with 

\begin{center}
\emph{a complete locally homogeneous Riemannian metric.}
\end{center}

Here \emph{locally homogeneous} means that any two points in the manifold have isometric neighbourhoods. There are in fact 8 possible types of such metrics, and we will describe them soon.

We now explain all the terminology and the results stated in detail. In the meanwhile we also introduce some fundamental concepts concerning 3-manifolds.

\begin{figure}
 \begin{center}
\centering
\labellist
\small\hair 2pt
\endlabellist
  \includegraphics[width = 6 cm]{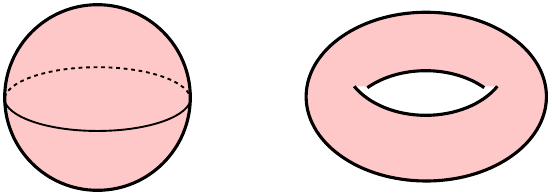}
 \end{center}
 \caption{A sphere and a torus. These are the two orientable connected compact surfaces without boundary with $\chi \geq 0$. Every 3-manifold is cut canonicaly along some spheres and tori.}  \label{surfaces:fig}
\end{figure}

\subsection{Three-manifolds and their submanifolds}
As already said above, every 3-manifold $M$ considered in this paper is tacitly assumed to be smooth, compact, orientable, and with (possibly empty) boundary that consists of a finite union of tori. The reason for admitting tori and not other kinds of surfaces is that tori fit very well in the geometrisation perspective. 
Note that we have $\chi(M)=0$. 

Every submanifold $N\subset M$ in a 3-manifold $M$ is tacitly assumed to be compact, without boundary, and contained in the interior of $M$. 
If $N$ is orientable, its tubular neighbourhood $\nu N$ is always trivial! If $N$ is a 1-manifold we get $\nu N \isom N \times \matR^2$, and if $N$ is an orientable surface we get $\nu N \isom N \times \matR$. 
Life is easier in dimension three, and much more complicated in dimension four, where orientable surfaces in orientable 4-manifolds may have (and often have) non-trivial tubular neighbourhoods: for instance, the tubular neighbourhood of a line in $\matCP^2$ is not trivial.

\subsection{Knots and links}
A 1-dimensional submanifold $L \subset M$ of a 3-manifold $M$ is called a \emph{link}. It is diffeomorphic to a finite union of circles, and if it consists of a single circle it is called a \emph{knot}. Mathematicians have started to classify knots in $S^3$ or $\matR^3$ up to isotopy since the XIX century, most notably when Tait started to produce tables like the one shown in Figure \ref{Knot_table:fig}. Considering knots in $S^3$ or $\matR^3$ is pretty much the same, and we prefer $S^3$ because it is compact.

\begin{figure}
 \begin{center}
\centering
  \includegraphics[width = 9 cm]{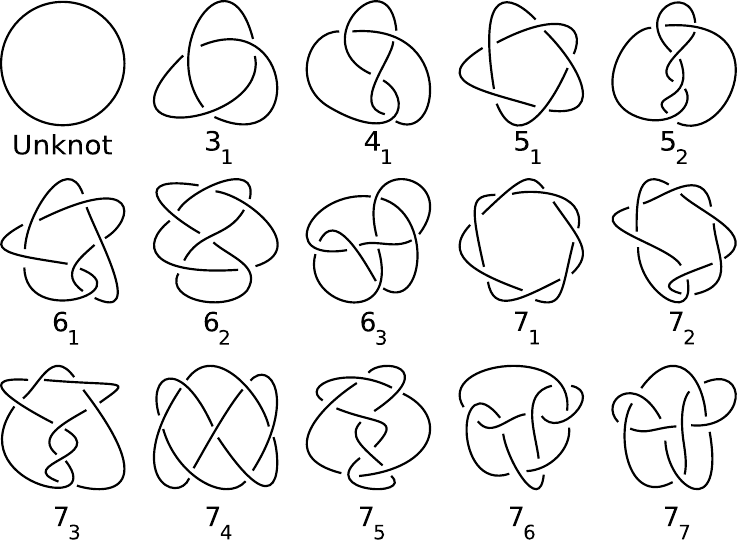}
 \end{center}
 \caption{The 15 knots in $S^3$ that can be described using a planar diagram with at most 7 crossings, considered up to isotopies and reflections. The knots $3_1$ and $4_1$ are called the \emph{trefoil knot} and the \emph{figure eight knot}. }  \label{Knot_table:fig}
\end{figure}

\begin{figure}
 \begin{center}
\centering
  \includegraphics[width = 3 cm]{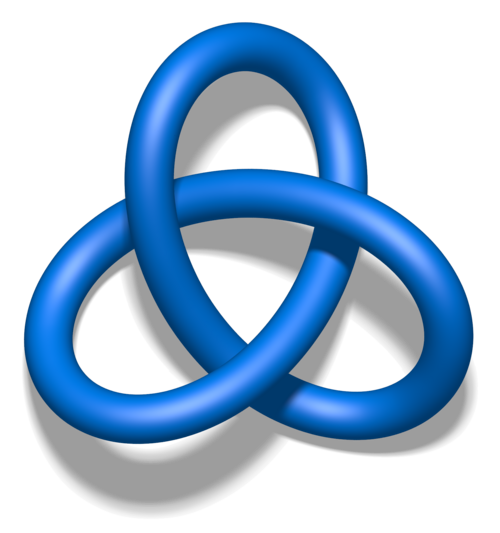}
\hspace{1 cm}
  \includegraphics[width = 3 cm]{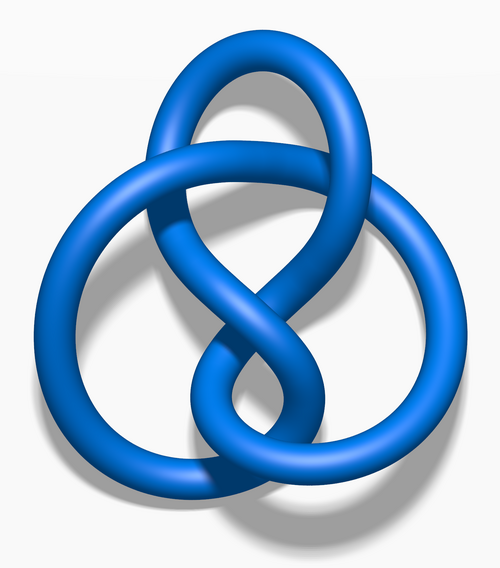}
   \end{center}
 \caption{The closed tubular neighbourhood of a knot is a (knotted) solid torus $S^1 \times D^2$. Here we show the closed tubular neighbourhoods of the trefoil and the figure eight knots. }  \label{knots:fig}
\end{figure}

By what said above, the tubular neighbourhood of any knot $K\subset M$ is diffeomorphic to $K \times \matR^2 \isom S^1 \times \matR^2$. It is sometimes preferable to consider the smaller \emph{closed tubular neighborhood} $S^1 \times D^2 \subset S^1 \times \matR^2$, that is compact and diffeomorphic to the \emph{solid torus} $S^1 \times D^2$. Here $D^k \subset \matR^k$ denotes the unit disc. The boundary of a solid torus is the torus $S^1 \times S^1$. See some examples in Figure \ref{knots:fig}.

\begin{figure}
 \begin{center}
\centering
  \includegraphics[width = 12 cm]{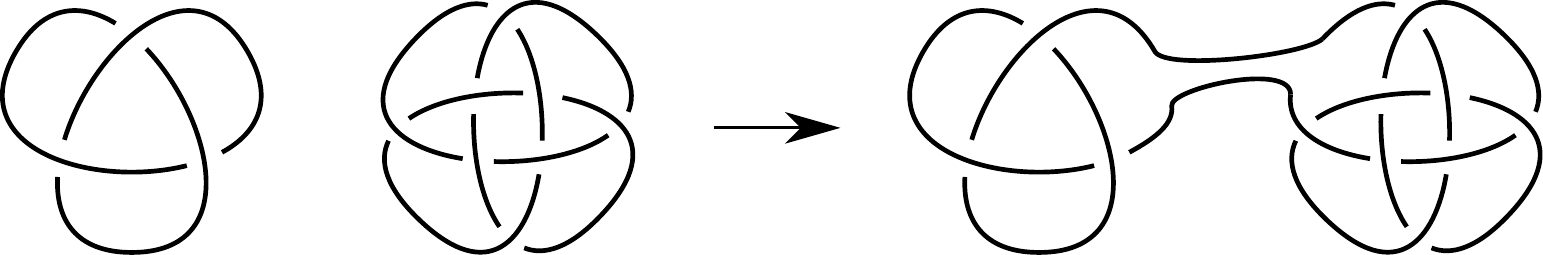}
 \end{center}
 \caption{The connected sum of two knots. }  \label{connected_sum:fig}
\end{figure}

The \emph{connected sum} of two knots is defined as in Figure \ref{connected_sum:fig}. If we consider oriented knots up to isotopy, this operation is well-defined and gives to this set the structure of a commutative monoid where the unknot is the identity element. A knot is \emph{prime} if it cannot be obtained as the connected sum of two non-trivial knots, and the \emph{prime factorization theorem} for knots ensures that every oriented knot may be described in a unique way as the connected sum of finitely many oriented prime knots: this is very much like the unique factorization theorem for natural numbers, and as a consequence when we study knots we typically restrict our attention to prime knots. The theorem was proved by Schubert \cite{Schu} in 1949.

Prime knots have been tabulated since the XIX century. The knots that can be described via a planar diagram with at most 7 crossings are those in Figure \ref{Knot_table:fig}. The two very recent last steps are Burton's classification of the prime knots up to 19 crossings \cite{Ben} and Thistlethwaite's improvement up to 20 crossings \cite{This}. The numbers of distinct (that is, not isotopic) prime knots are shown in Table \ref{knots:table}.

\begin{table}
\begin{center}
\begin{tabular}{c|ccccccccccccc}
$c$ & 3 & 4 & 5 & 6 & 7 & 8 & 9 & 10 & 11 & 12 & 13 & 14 & 15 \\
\hline
$n$ & 1 & 1 &	2	& 3	& 7	& 21	& 49	& 165	& 552	& 2176	& 9988	& 46972	& 253293	
\end{tabular}
\begin{tabular}{c|ccccc}
$c$ & 16 & 17 & 18 & 19 & 20 \\
\hline
$n$ & 1388705	& 8053393	& 48266466	& 294130458	& 1847319428
\end{tabular}
\end{center}
\vspace{.3 cm}
 \caption{The number $n$ of prime knots that can be described via a planar diagram with $c$ crossings, with $c\leq 20$.}  \label{knots:table}
\end{table}


The \emph{exterior} of a link $L\subset S^3$ is the 3-manifold $M = S^3 \setminus \nu L$ obtained by removing from $S^3$ a small open tubular neighbourhood $\nu L$ of $L$. The exterior $M$ has one torus boundary component for each component of $L$. By Alexander's duality we have $H_1(M) = \matZ^k$ where $k$ is the number of components of $L$. A famous theorem of Gordon and Luecke \cite{GL} ensures that two knots are isotopic if and only if they have diffeomorphic exteriors. This is not true for links with at least two components.

\subsection{Incompressible surfaces}
A 3-manifold $M$ contains many surfaces, that can be manipulated in various ways. Let $S \subset M$ be an orientable connected surface. Recall that by assumption $S$ is compact, without boundary, and contained in the interior of $M$. We have $\chi(S) = 2-2g$ where $g$ is the genus of $S$. We suppose that $g \geq 1$, so $S$ is not a sphere. 

If there is a disc $D$ in $M$ that intersects $S$ only in its boundary as in Figure \ref{incompressible:fig}-(left), we can \emph{compress} $S$ as in the figure and produce as a result a new surface $S'$ still contained in $M$. If $S$ is connected, the new surface $S'$ may have either one or two components. We have $\chi(S') = \chi(S) + 2$ in either case, and from this we deduce that each component of $S'$ has genus not larger than $g$. Moreover, if the genus of some component of $S'$ is exactly $g$, there are two components in $S'$, and the other component is a sphere: this case holds precisely when $\partial D$ bounds another disc $D'$ contained in $S$, as in Figure \ref{incompressible3:fig}.


\begin{figure}
 \begin{center}
\centering
  \includegraphics[width = 9 cm]{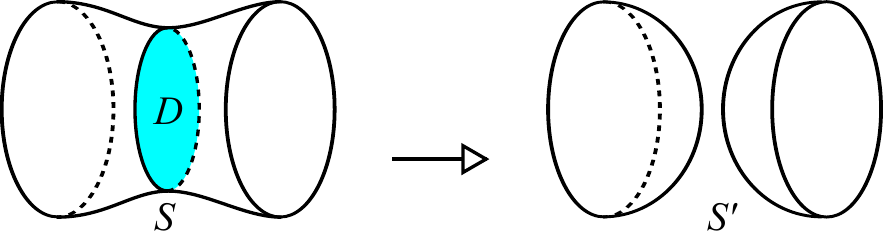} 
 \end{center}
 \caption{A compression of a surface. }  \label{incompressible:fig}
\end{figure}

\begin{figure}
 \begin{center}
\centering
  \includegraphics[width = 10 cm]{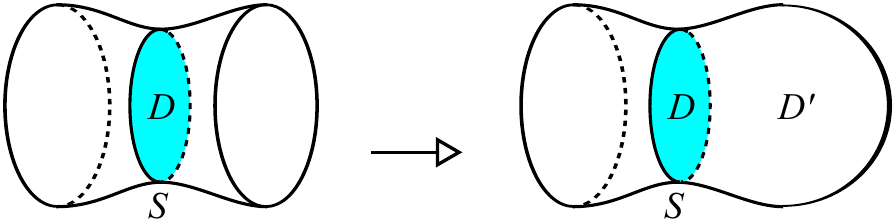} \\
 \end{center}
 \caption{An incompressible surface. }  \label{incompressible3:fig}
\end{figure}

We say that $S$ is \emph{incompressible} if it cannot be simplified by any compression. More precisely, we require that if there is a compressing disc $D$ as in Figure \ref{incompressible3:fig}-(left), then there is a disc $D' \subset S$ with $\partial D = \partial D'$ as in Figure \ref{incompressible3:fig}-(right). By what just said, this means that by compressing $S$ along any disc $D$ we would never get a surface $S'$ whose components have each genus smaller than $g$. There is no way to transform $S$ into simpler surfaces by any compression. 

The following is the first important theorem about 3-manifolds. Note that the inclusion $S \hookrightarrow M$ induces a homomorphism $\pi_1(S) \to \pi_1(M)$ of fundamental groups.

\begin{lemma}[Dehn's Lemma] An orientable connected surface $S \subset M$ of genus $g \geq 1$ is incompressible if and only if the map $\pi_1(S) \to \pi_1(M)$ is injective.
\end{lemma}

The Lemma was famously stated and ``proved'' by Dehn \cite{Dehn} in 1910. The proof however contained a serious gap that was then solved by Papakyriakopoulos \cite{Papa} in 1957. Dehn's lemma is in fact more subtle than it might seem! If we replace inclusion $S \subset M$ with any map, it transforms into a well-known old conjecture, that has been verified only for some classes of 3-manifolds:

\begin{conj}[Simple Loop Conjecture]
Let $f\colon S \to M$ be a map from an orientable surface to a 3-manifold. If $f_* \colon \pi_1(S) \to \pi_1(M)$ is not injective, there is an element in $\ker f_*$ that is represented by a simple (that is, embedded) loop.
\end{conj}

\subsection{Essential surfaces}
A 3-manifold $M$ contains plenty of surfaces, but only few of them are really interesting for us. An orientable connected surface $S \subset M$ is \emph{essential} if the following holds:

\begin{enumerate}
\item If $S$ is a sphere, it does not bound a three-dimensional disc $D^3$ in $M$;
\item If $S$ has genus $g\geq 1$, it is incompressible;
\item If $S$ is a torus, it is not isotopic to a component of $\partial M$.
\end{enumerate}

We have cited Dehn's lemma as the first important theorem on 3-manifolds. The following is arguably the second, proved by Alexander \cite{Alex} in 1924.

\begin{teo}[Alexander's Theorem]
There are no essential surfaces in $S^3$.
\end{teo}
\begin{proof}
Since $S^3$ is simply connected, by Dehn's lemma every surface $S$ of genus $g\geq 1$ is compressible. 
We are left with spheres, and this is in fact the content of Alexander's theorem: every sphere $S \subset S^3$ bounds a 3-dimensional disc $D^3$, and actually it bounds two discs, from both sides, simply because it is isotopic to a standard equatorial 2-sphere. Asking whether every codimension one sphere in $S^n$ is isotopic to a standard one is the famous \emph{Schoenflies problem}: this is the Jordan theorem in dimension $n=2$, and it is still a famous open problem in dimension $n=4$, so we should not expect an obvious proof in dimension $n=3$. However, the proof is not hard if you are sufficiently trained as a 3-manifold topologist.

\begin{figure}
 \begin{center}
\centering
  \includegraphics[width = 6 cm]{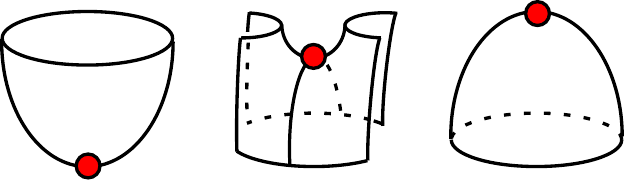}
 \end{center}
 \caption{The height function is a Morse function, with finitely many minima, saddle points, and maxima. }  \label{Morse:fig}
\end{figure}

Consider the surface $S$ in $\matR^3 = S^3 \setminus \{\infty\}$ and modify it with a generic rotation so that the height coordinate $z$ is a Morse function $z \colon S \to \matR$, with finitely many minimal, saddle points, and maxima as in Figure \ref{Morse:fig}, that are at different heights $z_1 < z_2 < \cdots < z_n$. Pick intermediate heights $z_1' < z_1 < z_2' < \cdots < z_n' < z_n < z_{n+1}'$ and cut the surface $S$ along the horizontal planes $P_i=\{z=z_i'\}$. Each plane $P=P_i$ cuts $S$ into a finite set of closed curves. In a neighbourhood of $P$, compress the surface $S$ starting from the innermost closed curves as shown in Figure \ref{Morse_cap:fig} (we use Jordan's theorem here to picture every closed curve as a circle). 

\begin{figure}
 \begin{center}
\centering
  \includegraphics[width = 10 cm]{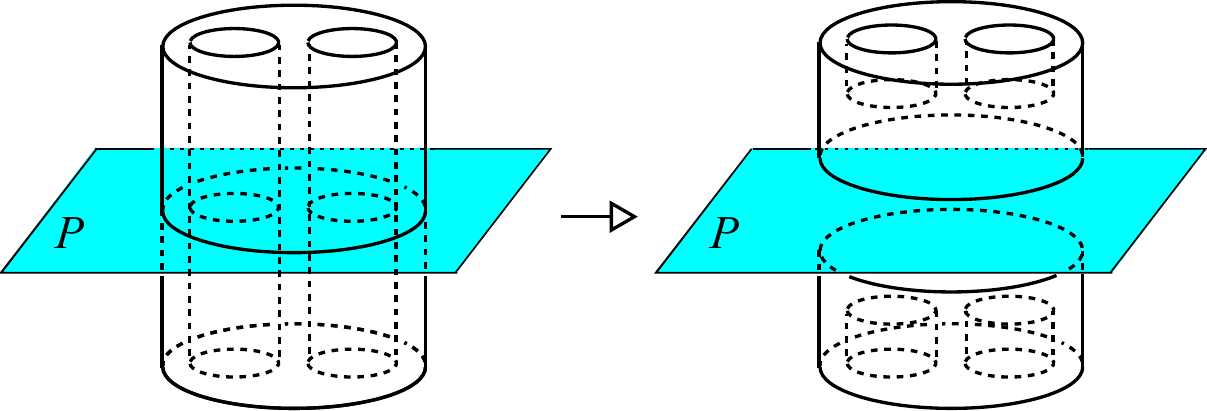}
 \end{center}
 \caption{For every horizontal plane $P$, compress $S$ starting from the innermost circles. }  \label{Morse_cap:fig}
\end{figure}

\begin{figure}
 \begin{center}
\centering
  \includegraphics[width = 9 cm]{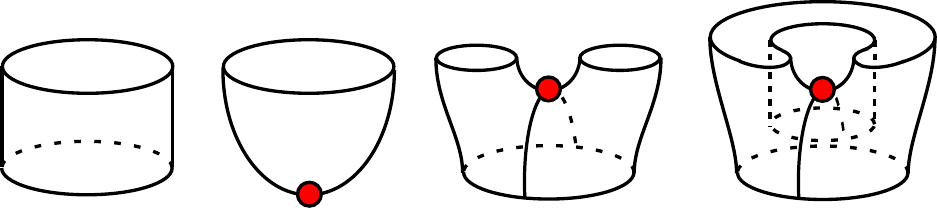}
 \end{center}
 \caption{After the multiple compressions, the resulting surface is a union of spheres as shown here. Each such sphere bounds a three-dimensional disc. }  \label{Morse2:fig}
\end{figure}

The result of these compressions is a new surface $S'$ that consists of many connected components, that are in fact all spheres as in Figure \ref{Morse2:fig}, and each such sphere bounds a three-dimensional disc, as the figure shows. By reversing the process, we deduce that the original sphere $S$ also bounds a disc: at each step, we glue two spheres that bound discs to a single sphere, that also bounds a disc, that is either the union or the difference of the two original discs.
\end{proof}

\subsection{Cutting along surfaces}

\begin{figure}
 \begin{center}
\centering
  \includegraphics[width = 9 cm]{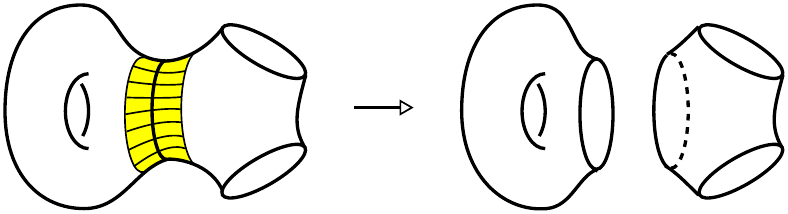}
 \end{center}
 \caption{Cutting a manifold along a codimension 1 submanifold with trivial tubular neighbourhood.}  \label{cut_new:fig}
\end{figure}

If $S \subset M$ is an orientable surface in a 3-manifold $M$, it has trivial tubular neighbourhood $\nu S \isom S \times \matR$, and we may \emph{cut} $M$ along $S$ as shown in Figure \ref{cut_new:fig}. The cut consists of replacing $S \times \matR$ with $S \times (-\infty,0] \sqcup S \times [0, +\infty)$. The result is a new 3-manifold $M'$ with boundary
$$\partial M' = \partial M \sqcup S^+ \sqcup S^-$$ 
where $S^+, S^-$ are both diffeomorphic to $S$. If $S$ is not a torus, $M'$ is not strictly speaking a 3-manifold of the type considered in this paper because it has non-toric boundary components.

By Alexander's theorem, if we cut $S^3$ along any two-sphere we get two discs.

\subsection{Prime decomposition}
There is a well-established notion of connected sum on 3-manifolds that is similar to the one on knots defined above. Let $M_1$ and $M_2$ be two connected oriented 3-manifolds. The \emph{connected sum} $M_1 \# M_2$ of $M_1$ and $M_2$ is defined as follows: 
\begin{enumerate}
\item We pick two 3-dimensional discs $D_1 \subset \interior {M_1}$ and $D_2 \subset \interior {M_2}$ and remove their interiors, thus adding two new boundary spheres $S_1 = \partial D_1$ and $S_2 = \partial D_2$ to the boundaries $\partial M_1$ and $\partial M_2$;
\item We glue these spheres via an orientation-reversing diffeomorphism $S_1 \to S_2$.
\end{enumerate}

Since the diffeomorphism is orientation-reversing, the new manifold $M_1\# M_2$ is oriented; we have $\partial (M_1\#M_2) = \partial M_1 \sqcup \partial M_2$.
By a theorem of Palais \cite{Palais} the discs $D_1, D_2$ are in fact unique up to ambient isotopy, and by a theorem of Smale \cite{Smale} the diffeomorphism $S_1 \to S_2$ is also unique up to isotopy; both facts imply that the resulting manifold $M_1 \# M_2$ is unique up to diffeomorphism. It is worth noting that Palais' theorem holds in every dimension $n$, while Smale's does not! For this reason, to define the connected sum in higher dimension we must choose the gluing diffeomorphism with more care. 

The manifold $M_1 \# M_2$ contains a sphere $S$, that is the result of glueing $S_1$ with $S_2$. We can invert the move   by cutting $M_1 \# M_2$ along $S$, and then capping off the two new resulting sphere boundary components with two discs $D_1$ and $D_2$. This is what we call \emph{cutting a manifold along a sphere}: we first cut along a sphere, thus producing two new spherical boundary components, and then attach two discs. 

We observe that the resulting sphere $S$ is essential, that is it does not bound a disc, if and only if both $M_1$ and $M_2$ are not diffeomorphic to $S^3$. 

Oriented 3-manifolds form a commutative monoid with the connected sum operation, with $S^3$ playing the role of the neutral element. A connected sum $M=M_1 \# M_2$ is \emph{non-trivial}, that is $M_1, M_2 \neq S^3$, if and only if the resulting sphere $S$ is essential. The presence of an essential sphere $S$ in a 3-manifold $M$ that separates $M$ into two components thus certifies that $M$ decomposes non-trivially as some connected sum $M=M_1 \# M_2$.

Like with knots, we say that an oriented 3-manifold different from $S^3$ is \emph{prime} if it cannot be obtained as the connected sum of two 3-manifolds, both different from $S^3$. By what just said this is equivalent to saying that $M$ contains no essential sphere $S$ that decomposes $M$ into two pieces.
After Dehn's lemma and Alexander's theorem, the third important result in three-dimensional topology is the following, where existence was proved by Kneser \cite{Kneser} and uniqueness by Milnor \cite{Mil}.

\begin{teo}[Prime Decomposition Theorem] \label{prime:teo}
Every 3-manifold $M\neq S^3$ decomposes uniquely as the connected sum of finitely many prime 3-manifolds.
\end{teo}
\begin{proof}[Proof of existence]
Let $M$ be a 3-manifold. To prove the $M$ decomposes as the connected sum of finitely many prime manifolds, the strategy is clear: if it is not prime, we have $M= M_1 \# M_2$ with $M_1,M_2 \neq S^3$, then we look at $M_1$ and $M_2$, and iterate. But how can we be sure that this process ends? 

We have remarked that $M= M_1 \# M_2$ with $M_1,M_2 \neq S^3$ precisely when $M$ contains an essential sphere $S\subset M$ that separates $M$ into two components. So a way to ensure that the process ends is to show that $M$ cannot contain arbitrarily many essential disjoint and pairwise non-parallel spheres (two spheres $S, S'$ are parallel if they cobound a piece diffeomorphic to $S^2 \times [0,1]$). Kneser proved this by introducing a formidable tool called \emph{normal surface theory}, that we briefly describe.

\begin{figure}
\begin{center}
\includegraphics[width = 6 cm] {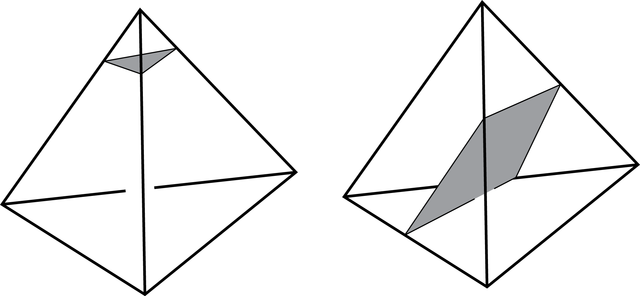}
\caption{A normal surface intersects every tetrahedron in triangles or squares.}
\label{normal_surface:fig}
\end{center}
\end{figure}

Pick a triangulation for $M$. A surface $S \subset M$ is \emph{normal} if it is transverse to the edges and faces of the triangulation, and it intersects every tetrahedron into some triangles and squares as in Figure \ref{normal_surface:fig}.
We prove that every surface $S \subset M$ without boundary becomes a normal surface after an isotopy and some \emph{elementary transformations}, that consist of (i) the removal of a component fully contained in a 3-disc and (ii) compressions along some discs as in Figure \ref{incompressible:fig}. 

The proof goes as follows. Put $S$ in transverse position with respect to the triangulation, and minimize its intersections with its edges and faces. After compressions as in Figure \ref{normal2:fig} we can suppose that the intersection of $S$ with each tetrahedron $\Delta$ consists of discs and components that are entirely contained in the interior of $\Delta$, that can be removed. So we get only discs. Consider one such disc $D$. If $\partial D$ is entirely contained in a triangular face of $\Delta$ as in Figure \ref{normal:fig} we remove it via a compression. If it intersects an edge at least twice, we remove it via an isotopy as in Figure \ref{normal3:fig}. We finally end up with discs $D$ whose boundary $\partial D$ intersect each edge at most one: these are triangles and squares, as asserted.

\begin{figure}
\begin{center}
\includegraphics[width = 12.5 cm] {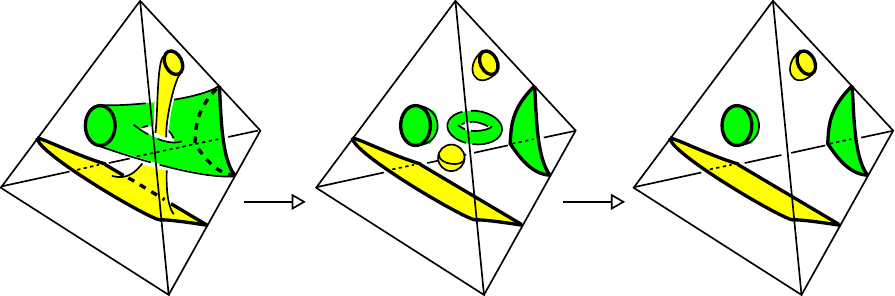}
\caption{After some compressions we get only discs and surfaces without boundary entirely contained in $\Delta$, that we remove.}
\label{normal2:fig}
\end{center}
\end{figure}

\begin{figure}
\begin{center}
\includegraphics[width = 6 cm] {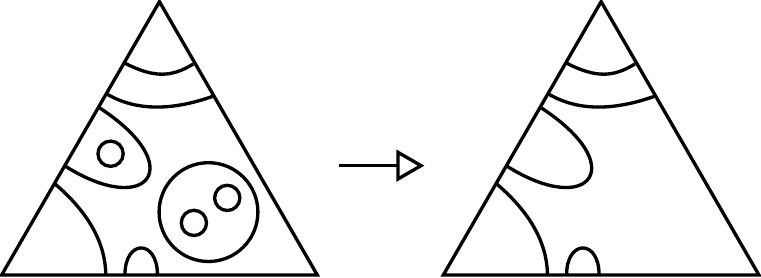}
\caption{We remove the circular intersections of $S$ with a triangle of the triangulation by compressions, starting with the innermost circles.}
\label{normal:fig}
\end{center}
\end{figure}

\begin{figure}
\begin{center}
\includegraphics[width = 8 cm] {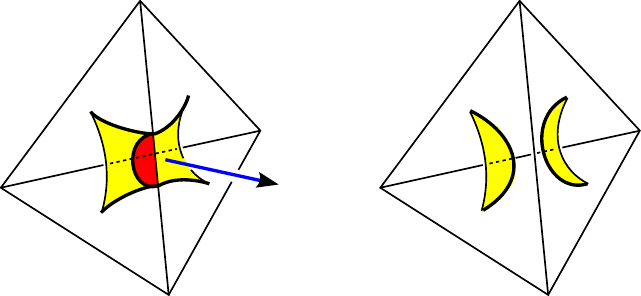}
\caption{If $\partial D$ intersects an edge twice, we remove both intersections with an isotopy (here again we start with innermost intersections).}
\label{normal3:fig}
\end{center}
\end{figure}

\begin{figure}
\begin{center}
\includegraphics[width = 9 cm] {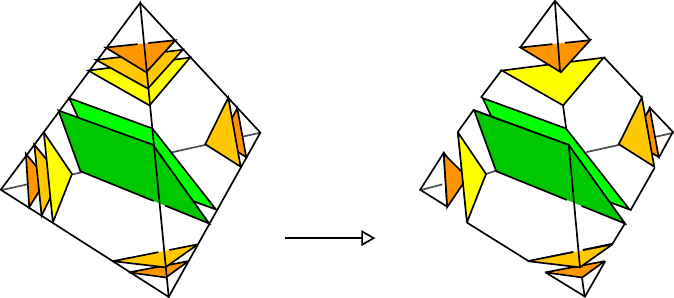}
\caption{By cutting $M$ along a normal surface $S$ we get a manifold that is obtained as the union of many prisms, and at most 6 other pieces for each tetrahedron (here we have 4 pyramids and two esahedra).}
\label{normal4:fig}
\end{center}
\end{figure}

If $S$ is a set of spheres, the resulting normal surface $S'$ is again a set of spheres. One proves that if the spheres in $S$ 
are essential and pairwise non-parallel, then $S'$ also contains a set of essential and pairwise non-parallel spheres with at least the same cardinality of $S$ (essentiality cannot be lost in the process). We keep these spheres in $S'$ and remove all the others.

We conclude by noting that a normal surface $S'$ may have at most $k$ pairwise non-parallel components, where $k$ is some number that depends only on the triangulation $\Delta$. This is shown essentially in Figure \ref{normal4:fig}: by cutting $M$ along $S'$ we get a manifold $M'$ that is made of many prisms, but only $6t$ non-prismatic polyhedra, where $t$ is the number of tetrahedra in $\Delta$. Therefore only at most $6t$ components of $M'$ are not made entirely of prisms. A manifold that is made entirely of prisms is either a product (excluded because the components of $S'$ are pairwise non-parallel), or a twisted bundle over a non-orientable surface, but there can be only at most $|H^1(M, \matZ/2\matZ)|$ of them since they all yield distinct non-trivial cohomology classes.

The proof of the existence of a prime decomposition is complete.
\end{proof}

\subsection{JSJ decomposition}
Having cut canonically every 3-manifold along spheres, we now turn to tori. This procedure is quite different from the previous one, and is usually called the \emph{JSJ decomposition}, following the names of 
Jaco, Shalen, and Johannson who discovered it independently \cite{JS, Joh} in 1979, precisely 50 years after Kneser's contribution. A different approach to the theory is given by Neumann and Swarup \cite{NS}.

Let $M$ be a prime 3-manifold. We say that an essential torus $T \subset M$ is \emph{canonical} if every other torus $T' \subset M$ can be isotoped to be disjoint from $T$. The following example shows that this is a serious requirement.

\begin{figure}
 \begin{center}
 \labellist
\small\hair 2pt
\pinlabel $\eta$ at 200 20
\pinlabel $\gamma$ at 100 50
\pinlabel $S$ at 220 5
\endlabellist
 \centering
  \includegraphics[width = 6 cm]{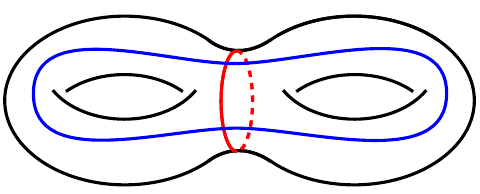}
 \end{center}
 \caption{Every essential curve $\gamma$ in a surface $S$ intersects some other essential curve $\eta$ in an essential way (there is no way to isotope $\eta$ away from $\gamma$).}  \label{surface:fig}
\end{figure}

\begin{example} \label{SxS1:example}
Let $S$ be a surface of genus $g \geq 1$ as in Figure \ref{surface:fig}.
An \emph{essential curve} in $S$ is a simple closed curve $\gamma \subset S$ that does not bound a disc. If $\gamma \subset S$ is an essential curve as in the figure, then $T=\gamma \times S^1 \subset S \times S^1$ is an essential torus in the 3-manifold $M=S \times S^1$. However, this essential torus is not canonical: it is always possible to find another essential curve $\eta \subset S$ that intersects $\gamma$ in an essential way, so that there is no way to isotope $T' = \eta \times S^1$ away from $T$. See Figure \ref{surface:fig}.
\end{example}

\begin{teo}[JSJ decomposition]
Let $M$ be a prime 3-manifold. Let $T = T_1 \sqcup \cdots \sqcup T_k \subset M$ be a maximal collection of disjoint and pairwise non-parallel canonical tori. This collection exists and is unique up to isotopy.
\end{teo}
\begin{proof}[Na\"\i ve sketch of the proof]
The existence of a maximal collection follows from normal surface theory: similarly as in the proof of Theorem \ref{prime:teo}, we can show that there is an upper bound on the number of essential pairwise non-parallel tori. Uniqueness seems easy to achieve: if you have another collection $T' = T_1', \ldots, T_{k'}$, then by assumption every $T_j'$ can be isotoped away from every $T_i$, and hence the two collections $T$ and $T'$ can be isotoped to be disjoint, and by maximality of $T$ every $T_j'$ is actually parallel to some $T_i$, so $T=T'$ after an isotopy.

However we have cheated here, since it is not obvious that one can isotope $T_j'$ away from all $T_1,\ldots, T_k$ \emph{simultaneously}. So the actual proof is harder than that.
\end{proof}

This maximal collection $T$ of canonical tori is called the \emph{JSJ decomposition} of $M$. Each torus $T_i$ has a trivial tubular neighbourhood and we can cut $M$ along the tori $T_1,\ldots, T_k$, producing a possibly disconnected manifold $M'$ such that 
$$\partial M' = \partial M \sqcup T_1^0 \sqcup T_1^1 \sqcup \cdots \sqcup T_k^0 \sqcup T_k^1,$$
where $T_i^0, T_i^1$ are two copies of $T_i$. Therefore $M'$ has the same boundary as $M$, plus $2k$ more additional tori.
The connected components $M_1,\ldots, M_k$ of $M'$ are the \emph{pieces} of the JSJ decomposition. Of course the JSJ decomposition may be empty: this occurs when $M$ contains no canonical tori at all, and in that case $M'=M$. One example is $S^3$, that does not contain essential tori at all, and another is the product manifold $S \times S^1$ from Example \ref{SxS1:example}, which contains many incompressible tori, but none of which is canonical. 

In the very specific case where the collection $T$ consists of only one torus, and the resulting pieces consist of either one or two line bundles, we actually remove the torus from the JSJ decomposition and leave it empty. 
Note that only two line bundles can occur: the trivial line bundle over the torus, and the non-trivial orientable line bundle over the Klein bottle. This very specific case may arise only if $M$ is a torus bundle over $S^1$, or is doubly covered by a torus bundle, and it is more convenient to impose an empty JSJ decomposition for such manifolds.

When we previously cut along spheres, we capped off the boundary components by attaching discs. Here we cut along tori, and there is no canonical way to attach any manifold to these new boundary tori $T_i^j$, so we just keep them as they are. The 3-manifold $M' = M_1\sqcup \cdots \sqcup M_k$ is obtained from the prime 3-manifold $M$ by \emph{cutting it along tori}. This operation is canonical. 

The main reason for allowing $M$ to have boundary tori arises from here: even if we consider only 3-manifolds $M$ with empty boundary, they may have a non-trivial JSJ decomposition into manifolds bounded by tori at the end, so allowing boundary tori from the very beginning comes at no additional cost.

\subsection{Summary of the decompositions}
We have explained in detail how every 3-manifold $M$ decomposes canonically, first via the prime decomposition, and then via the JSJ decomposition, into some pieces $M_1,\ldots, M_k$. Every piece $M_i$ is itself a 3-manifold. Even if $M$ has empty boundary, some of the resulting pieces may have boundary, that consists of tori. We can reverse the process, and reconstruct the original 3-manifold $M$ from these pieces by first gluing some of them along their toric boundary components, and finally by doing some connected sums. There are multiple ways to do this, so $M$ is not determined by the list of manifolds $M_1,\ldots, M_k$ alone, but also by the way these are assembled.

\subsection{Geometric pieces}
To understand the geometrisation theorem, we are left with the definition of a \emph{geometric piece}. The definition is very simple to state. 

\begin{defn} \label{geometric:defn}
A 3-manifold $M$ is \emph{geometric} if its interior admits a complete locally homogeneous Riemannian metric.
\end{defn}

Here \emph{locally homogeneous} means that every two points $x,y$ have isometric neighbourhoods. Thurston showed \cite{Thu} that only 8 possible geometries may arise:
$$\matH^3, \ \matR^3, \ S^3, \ S^2 \times \matR, \ \matH^2 \times \matR,\ \widetilde{\SL_2},\ \Nil,\ \Sol.$$

Each of these symbols indicates a homogeneous complete Riemannian simply connected 3-manifold. The term \emph{homogeneous} means as usual that the isometry group of the manifold acts transitively on points. The spaces $\matH^k, \matR^k, S^k$ are the standard simply connected complete spaces with constant sectional curvature $-1,0,1$. The constant curvatures in the remaining 5 spaces are not constant.

The last three spaces $\widetilde{\SL_2}, \Nil, \Sol$ are Lie groups $G$, and by assigning an arbitrary metric to $\mathfrak g = T_e G$ and extending it by left multiplication we get a homogeneous metric on $G$. The group $\widetilde{\SL_2}$ is the universal cover of $\SL_2(\matR)$. The group $\Nil$ is the Heisenberg group, that consists of all matrices of type
$$\begin{pmatrix} 1 & x & z \\ 0 & 1 & y \\ 0 & 0 & 1 \end{pmatrix}.$$
This is a nilpotent group, hence the name $\Nil$. The group $\Sol$ is again $\matR^3$ equipped with the operation
$$(x,y,z) \cdot (x',y',z') = (x+e^{-z}x', y + e^zy', z+z').$$
This group is solvable and not nilpotent, whence the name $\Sol$. 

Thurston proved \cite{Thu} that the locally homogeneous Riemannian metrics that may arise in Definition \ref{geometric:defn} are locally isometric to one of the 8 Riemannian manifolds just listed. Since the metrics are complete, for each geometric manifold $M$ we get 
$$\interior M = \matX^3 / \Gamma$$
where $\matX^3$ is one of the 8 Riemannian models listed and $\Gamma< \Isom(\matX^3)$ is a discrete subgroup acting freely on $\matX^3$. When $M$ has boundary, its interior is not compact and might have infinite volume; it turns out that except in very few simple cases (solid tori and line bundles) the volume of $M$ is always finite. This is notably true for hyperbolic 3-manifolds: at each topological end $T \times [0, +\infty)$ the metric is $e^{-2t} g^T \oplus 1$ at a point $(x,t)$, that is we have a flat torus $T$ with some flat metric $g^T$ that is shrunk exponentially fast in the $t$ direction. Such a portion is called a \emph{cusp} and it has finite volume, equal to half the area of $T$.

If $M$ has no boundary, it may admit at most one of the 8 geometries. If $M$ has boundary, different geometries may coexist. If $M$ is hyperbolic with finite volume, by Mostow and Prasad's rigidity theorem \cite{Mostow, Mostow2, Prasad} it admits only one hyperbolic metric up to isometry.
If $M$ is flat, like a 3-torus, it may admit uncountably many non-isometric flat metrics. 


A geometric 3-manifold $M$ may contain some essential torus only if the geometry is 
$\matR^3, \matH^2 \times \matR, \widetilde{\SL_2}, \Nil$, or $\Sol$. The essential torus is not canonical, since the JSJ decomposition for a geometric manifold is always empty. 

\subsection{Geometrisation}
We can finally state the geometrisation theorem.

\begin{teo} [Geometrisation]
Let $M$ be a prime 3-manifold. Every piece arising from the JSJ decomposition of $M$ is geometric.
\end{teo}

If the JSJ decomposition is empty, the manifold $M$ is itself geometric: this holds for instance if $M$ contains no essential tori. The theorem was conjectured by Thurston \cite{Thu} and proved by him in many cases, which include all prime 3-manifolds $M$ that contain at least one incompressible surface, and all prime 3-manifolds $M$ with boundary.
Perelmann then proved the theorem \cite{Per1, Per3, Per2} in full generality via completely different analytic methods. 
Geometrisation implies the famous Poincar\'e conjecture: 

\begin{cor}[Poincar\'e Conjecture]
Every simply connected 3-manifold $M$ without boundary is diffeomorphic to $S^3$.
\end{cor}
\begin{proof}
Recall that $M$ is compact by assumption. We can restrict to the case where $M$ is prime. Since $\pi_1(M)$ does not contain $\matZ \times \matZ$, by Dehn's lemma $M$ contains no essential torus, and hence the JSJ decomposition of $M$ is empty. Therefore $M = \matX^3/\Gamma$ is geometric, and since $M$ is simply connected and compact the only possibility is $\matX^3 = S^3$ and $\Gamma$ trivial.
\end{proof}

As promised in the introduction we also state a more focused version, that says that $M$ is hyperbolic unless there are no obvious obstructions; to get a cleaner statement we consider only 3-manifolds without boundary.

\begin{teo}[Hyperbolisation]
Every prime 3-manifold $M$ without boundary such that $\pi_1(M)$ is not finite and does not contain $\matZ \times \matZ$ is hyperbolic.
\end{teo}

\section{The geometric 3-manifolds}
Having settled all the decomposition theory, it is now due time to describe many examples of geometric 3-manifolds. We start with the constant curvature ones. Recall that every 3-manifold $M$ here is tacitly assumed to be connected, compact, orientable, and with (possibly empty) boundary made of tori.

\subsection{$\matR^3$} A geometric 3-manifold locally isometric to $\matR^3$ is called \emph{flat}, because its sectional curvature vanishes everywhere.
Hantsche and Wendt proved \cite{HW} in 1935 that there are precisely 6 flat 3-manifolds without boundary up to diffeomorphism. These can be obtained by pairing isometrically the faces of some Euclidean polyhedra as shown in Figure \ref{flat:fig}. The figure shows that the first five manifolds are torus bundles over the circle with monodromy 
$$\begin{pmatrix} 1 & 0 \\ 0 & 1 \end{pmatrix}, \ 
\begin{pmatrix} -1 & 0 \\ 0 & -1 \end{pmatrix}, \ 
\begin{pmatrix} 0 & -1 \\ 1 & 0 \end{pmatrix}, \ 
\begin{pmatrix} 1 & -1 \\ 1 & 0 \end{pmatrix}, \ 
\begin{pmatrix} -1 & 1 \\ -1 & 0 \end{pmatrix}
$$
of order $1, 2, 4, 6$, and $3$ respectively. Each such matrix identifies an automorphism of the torus $T = \matR^2 /\matZ^2$, and the torus bundle is constructed by glueing the two boundary components of $T \times [0,1]$ via this automorphism.

The sixth manifold is called the \emph{Hantzsche-Wendt manifold}, and it does not fibre over $S^1$ because its first homology group is finite. This fascinating flat 3-manifold has been considered recently by some cosmologists \cite{cosmo}, and may be constructed elegantly by taking a rhombic dodecahedron as in Figure \ref{rhombic:fig} and pairing its opposite faces with a $\pi$ turn.

\begin{figure}
 \begin{center}
 \centering
  \includegraphics[width = 8 cm]{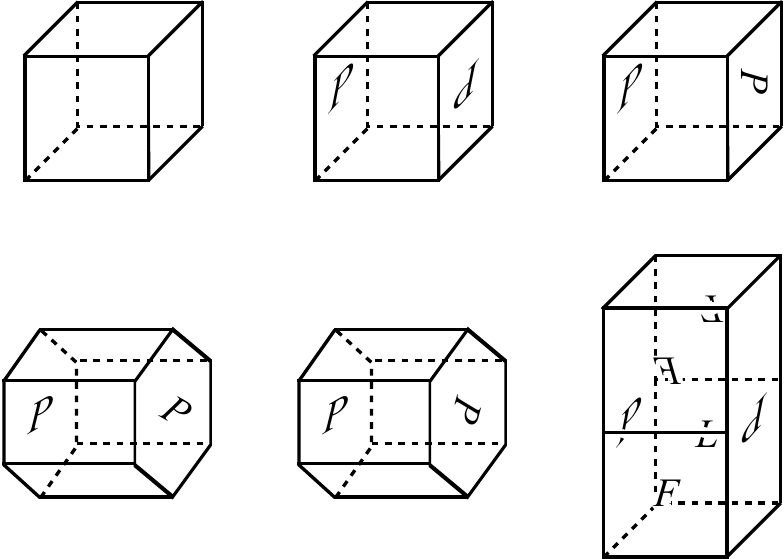}
 \end{center}
 \caption{The six flat orientable 3-manifolds are obtained from these Euclidean polyhedra by pairing their faces isometrically so that similar letters match. Each unlabeled face is paired with the opposite one with a translation: in the top-left cube every face is paired with the opposite one via a translation and we get the 3-torus; in the subsequent four polyhedra all the opposite faces are paired via a translation except the one marked with a $P$; the last one is more complicated.}  \label{flat:fig}
\end{figure}

\begin{figure}
 \begin{center}
 \centering
  \includegraphics[width = 3 cm]{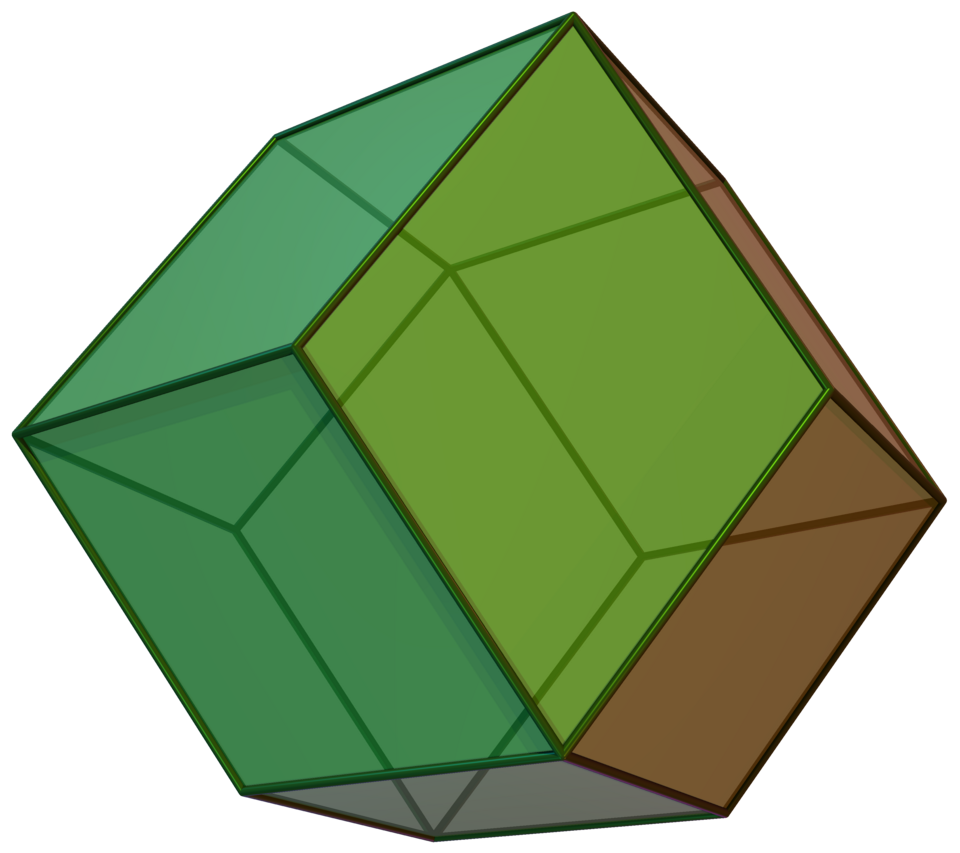}
 \end{center}
 \caption{By pairing the opposite faces of a rhombic dodecahedron with a $\pi$ rotation we get the Hantsche -- Wendt flat 3-manifold.}  \label{rhombic:fig}
\end{figure}

\subsection{$\matH^3$}
A geometric 3-manifold locally isometric to $\matH^3$ is called \emph{hyperbolic}.
The first hyperbolic 3-manifolds constructed in history were built out of regular hyperbolic polyhedra. When considered in Euclidean space, the five regular polyhedra in Figure \ref{regular:fig} have a fixed dihedral angle (the same angle on all edges), regardless of their diameter; only the cube has a nice dihedral angle that divides $2\pi$ and hence can be used to build some flat 3-manifolds as in Figure \ref{flat:fig}.

In $\matH^3$ and $S^3$ there are no similarities, and if we consider the regular polyhedra there, their dihedral angles (still the same on all edges) are not fixed and vary continuously and monotonically with the diameter; for this reason there are more regular polyhedra in $\matH^3$ and in $S^3$ than in $\matR^3$ whose dihedral angles divide $2\pi$. The angles of type $2\pi/k$ that can be obtained in $\matH^3$ for each regular polyhedron are shown in Figure \ref{regular:fig}. The ones marked in red are realised by \emph{ideal polyhedra}, that is polyhedra whose vertices all lie at infinity. These polyhedra are not compact, but they still have finite volume.

\begin{figure}
 \begin{center}
 \centering
\labellist
\small\hair 2pt
\pinlabel $\color{red}\pi/3$ at 400 -100
\pinlabel $\color{red}\pi/2$ at 1200 -100
\pinlabel $2\pi/3$ at 2000 -100
\pinlabel ${\color{red}\pi/3},2\pi/5$ at 2800 -100
\pinlabel ${\color{red}\pi/3},2\pi/5,\pi/2$ at 3700 -100
\endlabellist
  \includegraphics[width = 12.5 cm]{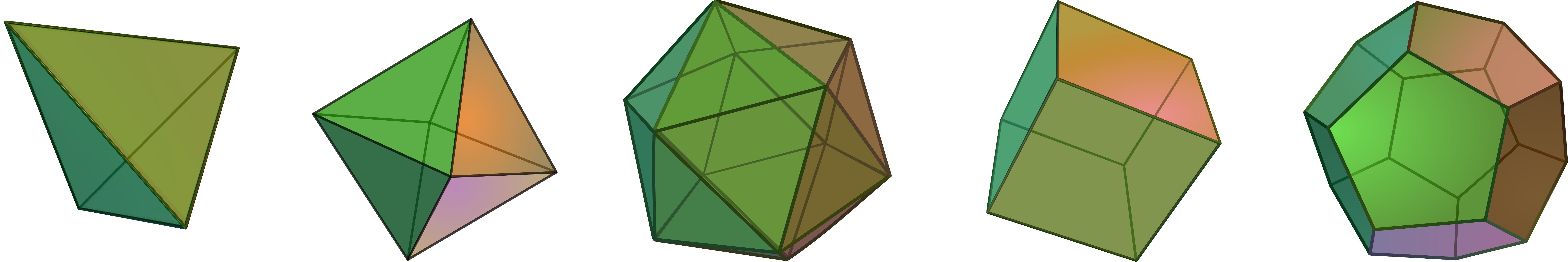}
  \vspace{.4 cm}
 \end{center}
 \caption{The five regular polyhedra may be represented in $\matH^3$ with the indicated dihedral angles. When the angles are red, the realization is via an ideal regular polyhedron.}  \label{regular:fig}
\end{figure}

\begin{figure}
 \begin{center}
 \centering
  \includegraphics[width = 3.5 cm]{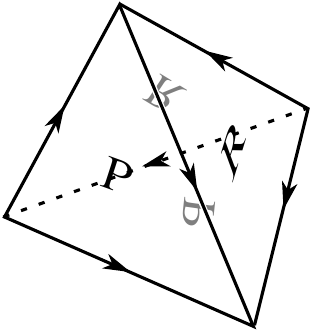}
 \end{center}
 \caption{The Gieseking hyperbolic 3-manifold is obtained by pairing the faces of a single regular ideal tetrahedron as shown here (match the letters). We can check that the pairing identifies all the 6 edges as indicated by the arrows.}  \label{Gieseking:fig}
\end{figure}

Gieseking constructed a non-compact finite volume non-orientable hyperbolic 3-manifold out of a single hyperbolic \emph{ideal regular tetrahedron} with dihedral angle $\pi/3$ in 1912, by pairing its faces isometrically as in Figure \ref{Gieseking:fig} \cite{Gies}. 
The reader can check that the matching identifies all the six edges to a single one; since the dihedral angle is $\pi/3$ we get $6\pi/3 = 2\pi$ and hence we indeed get a hyperbolic manifold: the figure is a complete proof! This non-orientable hyperbolic 3-manifold, now known as the \emph{Giesking manifold}, is the interior of a compact non-orientable 3-manifold having the Klein bottle as its boundary. Much later on, during Thurston's revolution in the 1970s, it was noted that  the orientable double cover of the Gieseking manifold is the complement in $S^3$ of the figure-eight knot shown in Figure \ref{knots:fig}.

The second hyperbolic 3-manifold in history was discovered by L\"obell \cite{Lobell} in 1931, inventing a technique that is still employed today to construct hyperbolic manifolds out of right-angled polyhedra in various dimensions, see for instance \cite{IMM}. L\"obell picked eight copies of the hyperbolic \emph{right-angled compact dodecahedron} and glued them along some face pairing. Two years later, Seifert and Weber built another compact hyperbolic 3-manifold out of a regular dodecahedron \cite{SW}, now with dihedral angles $2\pi/5$, simply by identifying every pair of opposite faces with a $3\pi/5$ twist. The compact dodecahedron with angles $2\pi/5$ is larger than the right-angled one, but smaller than the ideal one, that has angles $\pi/3$ (angles decrease when diameters increase).

More examples appeared only 40 years later. In the 1970s Riley started to construct hyperbolic structures on some knot complements in $S^3$ by looking at representations of their fundamental groups in $\PSL_2(\matC) = \Isom^+(\matH^2)$, and few years later Thurston entered the subject and rivolutionized it by showing that hyperbolic 3-manifolds are actually everywhere \cite{bibbia}. 

\subsection{$S^3$} A geometric 3-manifold locally isometric to $S^3$ is called \emph{spherical}. A spherical 3-manifold is the quotient $S^3 / \Gamma$ for some finite group $\Gamma < \SO(4)$ acting freely. Some notable examples are the \emph{lens spaces} $L(p,q)$, where $(p,q)$ is a coprime pair of natural numbers and $\Gamma$ is generated by a simultaneous rotation along two orthogonal planes of angles $2\pi/p$ and $2q\pi/p$. Lens spaces are called in this way because they can be realised geometrically by identifying the opposite faces of a \emph{lens} in $S^3$ with dihedral angle $2\pi/p$ via a $2\pi/q$ turn (a lens is a portion of $S^3$ delimited by two geodesic discs sharing the same closed geodesic as in Figure \ref{lens:fig}).

\begin{figure}
 \begin{center}
\centering
  \includegraphics[width = 4 cm]{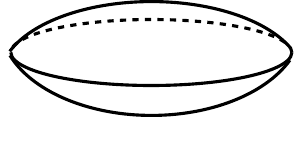}
 \end{center}
 \vspace{-.7 cm}
 \caption{A spherical lens.}  \label{lens:fig}
\end{figure}

Lens spaces are especially interesting
because, despite their relative simplicity, they display some instances of ``exoticness'' in dimension three: among lens spaces one can find homotopically equivalent 3-manifolds that are not diffeomorphic (like $L(7,1)$ and $L(7,2)$), and 3-manifolds with the same (cyclic) fundamental group that are not homotopically equivalent (like $L(5,1)$ and $L(5,2)$). As opposite to dimension four where all kinds of wild exoticness may arise (famously including manifolds that are homeomorphic and not diffeomorphic, something that is impossible in dimension three), in dimension three quite ironically these scary phenomena arise essentially only among lens spaces, and have no measurable effects on the other 3-manifolds. These facts were proved by Alexander \cite{Alexa}, Whitehead \cite{White}, and Reidemeister \cite{Reide} who introduced an invariant now called \emph{Reidemeister torsion} to distinguish between homotopy equivalence and diffeomorphism.

Another famous elliptic 3-manifold is \emph{Poincar\'e's homology sphere}, that can be realised as $S^3/I_{120}^*$ where $S^3$ is the group of unit quaternions and $I_{120}^*$ is the binary icosahedral group, or more geometrically by identifying the opposite faces of the \emph{spherical regular dodecahedron} with dihedral angles $2\pi/3$ with a $\pi/10$ turn, as discovered originally in \cite{SW}.

\subsection{$S^2 \times \matR$} This is the poorest geometry, since it concerns only two 3-manifolds, that are $S^2 \times S^1$ and its quotient by the involution $(x,e^{\theta i}) \mapsto (-x, e^{-\theta i})$.

\subsection{$\matH^2 \times \matR$} This geometry includes all the products $\Sigma \times S^1$ where $\Sigma$ is some hyperbolic surface, plus all the 3-manifolds that are covered by one of these. For instance, for every isometry $\varphi$ of some hyperbolic surface $\Sigma$ we can take $\Sigma \times [0,1]$ and then glue the two boundary components via $\varphi$. The isometry $\varphi$ has necessarily finite order and hence the resulting manifold is finitely covered by $\Sigma \times S^1$.

\subsection{$\widetilde \SL_2$} This geometry includes all the non-trivial circle bundles over surfaces $\Sigma$ with $\chi(\Sigma) < 0$, plus all the 3-manifolds that are covered by some of these. One example is the unit tangent bundle of a hyperbolic surface $\Sigma$.

\subsection{$\Nil$} This geometry includes all the non-trivial circle bundles over the torus, plus all the 3-manifolds that are covered by one of these. These circle bundles are also obtained as torus bundles over the circle with monodromy $\matr 1n01$ with $n\neq 0$.

\subsection{$\Sol$} This geometry includes all the torus bundles over the circle with monodromy $A \in \SL_2\matR$
having $|\tr A| > 2$, and the manifolds covered by these.

\section{Geometrisation in action}
We show the effect of geometrisation in a few different contexts.

\subsection{Knots}
By Alexander's theorem, the exterior $M = S^3 \setminus \nu K$ of a knot $K\subset S^3$ is a prime 3-manifold. However, it may have a non-trivial JSJ decomposition: this holds precisely when the knot $K$ is a \emph{satellite}, that is it is contained non-trivially in the closed tubular neighbourhood $\nu K'$ of another non-trival knot $K'$ (contained non-trivially means that $K$ is neither isotopic to $K'$ nor contained in a ball), see Figure \ref{satellite:fig}. In this case the boundary torus $\partial \nu K'$ is essential. Every non-prime knot is a particular kind of satellite knot.

\begin{figure}
 \begin{center}
\centering
  \includegraphics[width = 5 cm]{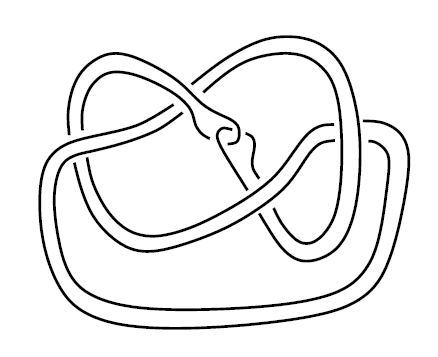}
 \end{center}
 \vspace{- .5 cm}
 \caption{A satellite knot is a knot that is contained non-trivially in the tubular neighbourhood of another non-trival knot.}  \label{satellite:fig}
\end{figure}

By geometrisation there are three types of non-trivial knots in $S^3$:
\begin{itemize}
\item the \emph{hyperbolic knots}, whose exterior is hyperbolic;
\item the \emph{torus knots}, whose exterior has geometry $\widetilde{\SL_2}$ or $\matH^2 \times \matR$;
\item the \emph{satellite knots}, whose exterior has a non-trivial JSJ decomposition.
\end{itemize}

A \emph{torus knot} is a knot contained in a standard torus: such knots are easily parametrized by pairs of coprime integers $(p,q)$. The exterior of a torus knot admits two different geometries because $\widetilde \SL_2$ and $\matH^2 \times \matR$ coexist in manifolds with non-empty boundary. This trichotomy appeared for the first time in \cite{Thu}. 

Hyperbolisation can be written succinctly as follows:

\begin{center}
\emph{Every knot that is neither a torus knot nor a satellite knot is hyperbolic.}
\end{center}

\subsection{Surface bundles}
Let a 3-manifold $M$ fibre over the circle. The fibre $\Sigma$ is a connected orientable surface with (possibly empty) boundary, and by fixing a horizontal vector field and taking its flow one finds a \emph{monodromy} $\varphi \colon \Sigma \to \Sigma$ that is well-defined only up to isotopy. We can recover $M$ from $\varphi$ by taking $\Sigma \times [0,1]$ and glueing its two boundary components via $\varphi$. We suppose that $\chi(\Sigma)<0$ since the finitely many cases $\chi(\Sigma) \geq 0$ are simpler to deal with.

Hyperbolisation in this context is easy to state:

\begin{center}
\emph{If $\varphi$ does not fix any finite set of essential curves up to isotopy, $M$ is hyperbolic.}
\end{center}

This formulation says as usual that $M$ is hyperbolic unless there is some obvious obstruction. A simple closed curve in $\Sigma$ is \emph{essential} if it does not bound a disc and it is not parallel to a boundary component of $\Sigma$. 
If $\varphi$ fixes a finite set of isotopy classes of essential curves, by sliding them along the fibration we would construct some (immersed) essential tori, and this would prevent $M$ from being hyperbolic. 

\begin{figure}
 \begin{center}
\centering
  \includegraphics[width = 12 cm]{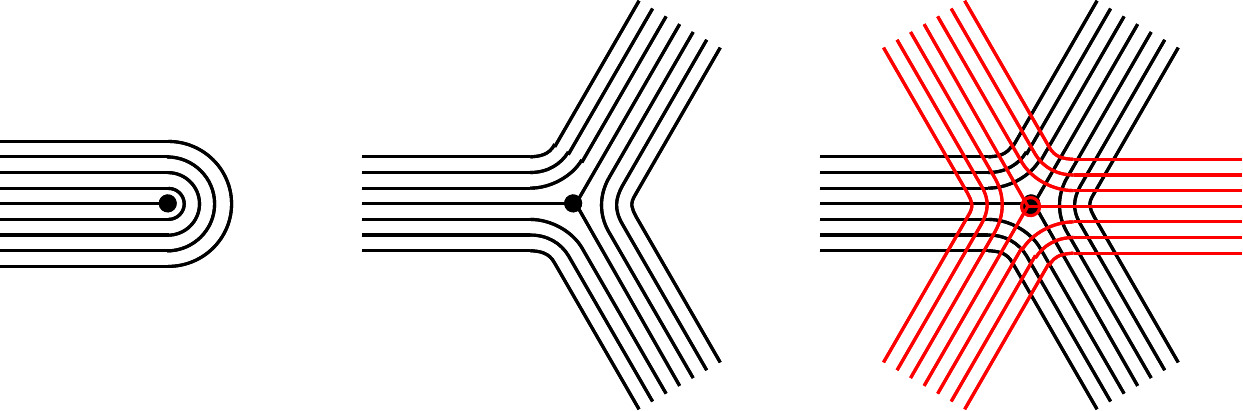}
 \end{center}
 \caption{A geodesic foliation near a singular point with angle $\pi$ (left) or $3\pi$ (center) in a flat cone surface. The angle $\pi$ is allowed only at ideal points. A metrically inaccurate sketch of two orthogonal geodesic foliations (right).}  \label{foliation:fig}
\end{figure}

Geometrisation can be formulated in a more structured way following Thurston's characterisation of the homeomorphisms of surfaces \cite{Thus}. Thurston proved that every self-homeomorphism $\varphi \colon \Sigma \to \Sigma$ can be chosed after an isotopy to be one of the following types:

\begin{enumerate}
\item \emph{finite order}, that is $\varphi^k=\id$ for some $k$;
\item \emph{reducible}, that is $\varphi$ fixes an essential curve;
\item \emph{pseudo-Anosov}.
\end{enumerate}

The types (1) and (2) are the only two that may hold simultaneously.
A \emph{pseudo-Anosov} map is a marvelous package that consists of:
\begin{itemize}
\item A \emph{cone flat structure} on the surface $\bar \Sigma$ obtained by shrinking the boundary components of $\Sigma$ to ideal points, where some points have cone angle $k\pi$ for some integer $k\neq 2$, and only the ideal points are allowed to have $k=1$;
\item Two orthogonal \emph{geodesic foliations} in $\bar \Sigma$ called \emph{stable} and \emph{unstable}, that look like Figure \ref{foliation:fig} near the singular points;
\item A homeomorphism $\varphi \colon \bar\Sigma \to \bar \Sigma$ that preserves the singular points and each foliation, and looks like an affine map away of the singular points that stretches the unstable leaves by some $\lambda > 1$ and contracts the stable ones by $1/\lambda$ (in particular it preserves areas). The number $\lambda > 1$ is called the \emph{stretching factor} of $\varphi$.
\end{itemize}

A beautiful and extensive introduction to pseudo-Anosov maps is contained in Farb -- Margalit \cite{FM}. Let $M$ be a 3-manifold fibering over the circle with monodromy $\varphi$.
The geometrisation type for $M$ depends elegantly on the type (1), (2), or (3) of $\varphi$ stated above, as follows. The manifold $M$ is always prime, and correspondingly to the type of $\varphi$ we have

\begin{enumerate}
\item $M$ has geometry $\matH^2 \times \matR$;
\item $M$ contains an essential torus;
\item $M$ is hyperbolic.
\end{enumerate}

Coherently with what said above only (1) and (2) may coexist. In case (2) we know that $\varphi$ fixes an essential curve $\gamma$, and to recover the full geometric decomposition of $M$ one should cut $\Sigma$ along $\gamma$ and proceed inductively. 

The correspondence between the three types of $\varphi$ and the three possible geometric outcomes for $M$ is among the most beautiful manifestations of geometrisation of 3-manifolds, and was proved by Thurston \cite{Thuf} before Perelman's more general accomplishment. A standard complete reference is Otal \cite{O}. 
The strong part in this correspondence is 
\begin{center}
\emph{$M$ hyperbolic $\Longleftrightarrow$ $\varphi$ pseudo-Anosov}.
\end{center}

It is worth noting that neither of the two implications is easy to prove: we cannot jump directly from the pseudo-Anosov map $\varphi$ to the hyperbolic structure on $M$, nor conversely from the hyperbolic structure of $M$ to the pseudo-anosov map $\varphi$.

One may wonder how general this construction is, since \emph{a priori} there could be only very few 3-manifolds that fibre over the circle. A recent celebrated breakthrough of Agol \cite{Agol}, obtained with fundamental contribution of Wise \cite{Wise}, shows that there are plenty of them, since (answering a famous question of Thurston \cite{Thu}) every hyperbolic 3-manifold has a finite cover that fibres over the circle.

\subsection{Polyhedra}
The first geometrisation of some 3-dimensional objects in history is Andreev's theorem \cite{Andreev, Andreev2}, published in 1970, that says that 

\begin{center}
\emph{Every abstract polyhedron with some non-obtuse angles assigned to its edges \\ is hyperbolic unless there are some obvious obstructions.}
\end{center}

\begin{figure}
 \begin{center}
\centering
\labellist
\small\hair 2pt
\pinlabel $\alpha_1$ at -10 70
\pinlabel $\alpha_2$ at 30 40
\pinlabel $\alpha_3$ at 95 95
\pinlabel $\alpha_1$ at 110 70
\pinlabel $\alpha_2$ at 150 40
\pinlabel $\alpha_3$ at 255 60
\pinlabel $\alpha_4$ at 215 95
\pinlabel $\alpha_1$ at 270 70
\pinlabel $\alpha_2$ at 310 40
\endlabellist
  \includegraphics[width = 10 cm]{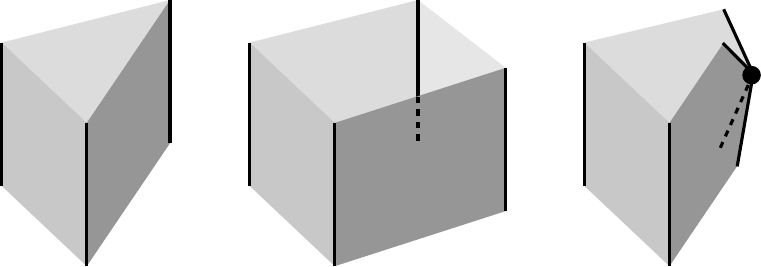}
 \end{center}
 \caption{Some configurations of faces in a polyhedron $P$. In the left (center) figure we suppose that the 6 (8) endpoints of the 3 (4) edges cointaining the labels $\alpha_i$ are all distinct (these configurations are called a \emph{3-circuit} and a \emph{4-circuit}). In the right figure we suppose that the left face does not contain the right vertex.} \label{Andreev:fig}
\end{figure}
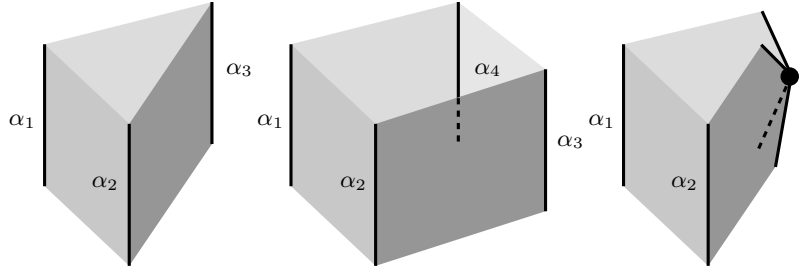

The original proof of Andreev contained a gap that was then fixed by Roeder \cite{Roeder}. Here we consider an abstract polyhedron $P$ with vertices of valence 3 and 4, different from a tetrahedron and a prism, and we assign abstract non-obtuse dihedral angles $0 < \alpha_i \leq \pi/2$ to its edges $e_i$. Andreev and Roeder say that $P$ can always be realised (uniquely up to isometries) as a honest finite-volume polyhedron in $\matH^3$ with these prescribed dihedral angles (with the 4-valent vertices at infinity), except when at least one of these obstructions occur:

\begin{enumerate}
\item $\alpha_i+\alpha_j+\alpha_k < \pi$ at some 3-valent vertex;
\item $\alpha_i +\alpha_j+\alpha_k+\alpha_l < 2\pi$ at some 4-valent vertex;
\item $\alpha_1 + \alpha_2 + \alpha_3 \geq \pi$ for some triple of faces as in Figure \ref{Andreev:fig}-(left);
\item $\alpha_1 + \alpha_2 + \alpha_3 +\alpha_4 \geq 2\pi$ for some quadruple of faces as in Figure \ref{Andreev:fig}-(center);
\item $\alpha_1+\alpha_2 \geq \pi$ for some triple of faces as in Figure \ref{Andreev:fig}-(right).
\end{enumerate}

All these obstructions are obvious in the sense that one can easily prove that they cannot arise on a hyperbolic polyhedron. The obstructions (2), (4), (5) are equivalent to $(\alpha_i, \alpha_j, \alpha_k, \alpha_l) \neq (\pi/2,\pi/2,\pi/2,\pi/2)$, $(\alpha_1,\alpha_2,\alpha_3,\alpha_4) = (\pi/2,\pi/2,\pi/2,\pi/2)$, $(\alpha_1,\alpha_2) = (\pi/2,\pi/2)$ respectively.

Thurston used this theorem in his proof of geometrisation for prime 3-manifolds that contain an incompressible surface (such 3-manifolds are called \emph{Haken}). When all the angles involved divide $\pi$, the polyhedron $P$ is a \emph{Coxeter polyhedron} and can be used to produce hyperbolic manifolds that are tessellated into copies of $P$.

\subsection{Branched coverings}
Let $L\subset S^3$ be a link with $k$ components. By Alexander's duality the group $H_1(S^3 \setminus L) = \matZ^k$ is generated by some small closed oriented curves encircling the components of $L$, called \emph{meridians}. We can construct a double covering of $S^3 \setminus L$ from the map $H_1(S^3 \setminus L) \to \matZ/2\matZ$ that sends all these meridians to 1, and extend it to a \emph{branched double covering} $M \to L$ by adding a copy of $L$ to the covering space. The 3-manifold $M$ is without boundary and is called the \emph{double branched covering} of $S^3$ ramified above $L$. As $L$ varies, with this method we get plenty of 3-manifolds $M$ whose topology can often be guessed by looking at $L$. This gives a useful visual way to build and study many 3-manifolds $M$ without boundary. Geometrisation as usual says:

\begin{center}
\emph{
The branched double covering of $S^3$ ramified over $L$ is hyperbolic \\ unless there are some obvious obstructions.
}
\end{center}

The obvious obstructions are:
\begin{enumerate}
\item $L$ is split;
\item $L$ is not prime;
\item $L$ contains a Conway sphere; 
\item $S^3 \setminus L$ contains an essential torus;
\item $L$ is a torus link;
\item $L$ is a Montesinos link of length $k\leq 3$.
\end{enumerate}

\begin{figure}
 \begin{center}
\centering
\labellist
\small\hair 2pt
\pinlabel $(1)$ at 350 480
\pinlabel $(2)$ at 760 480
\pinlabel $(3)$ at 350 300
\pinlabel $(4)$ at 540 300
\pinlabel $(5)$ at 760 300
\pinlabel $(6)$ at 760 0
\endlabellist
  \includegraphics[width = 10 cm]{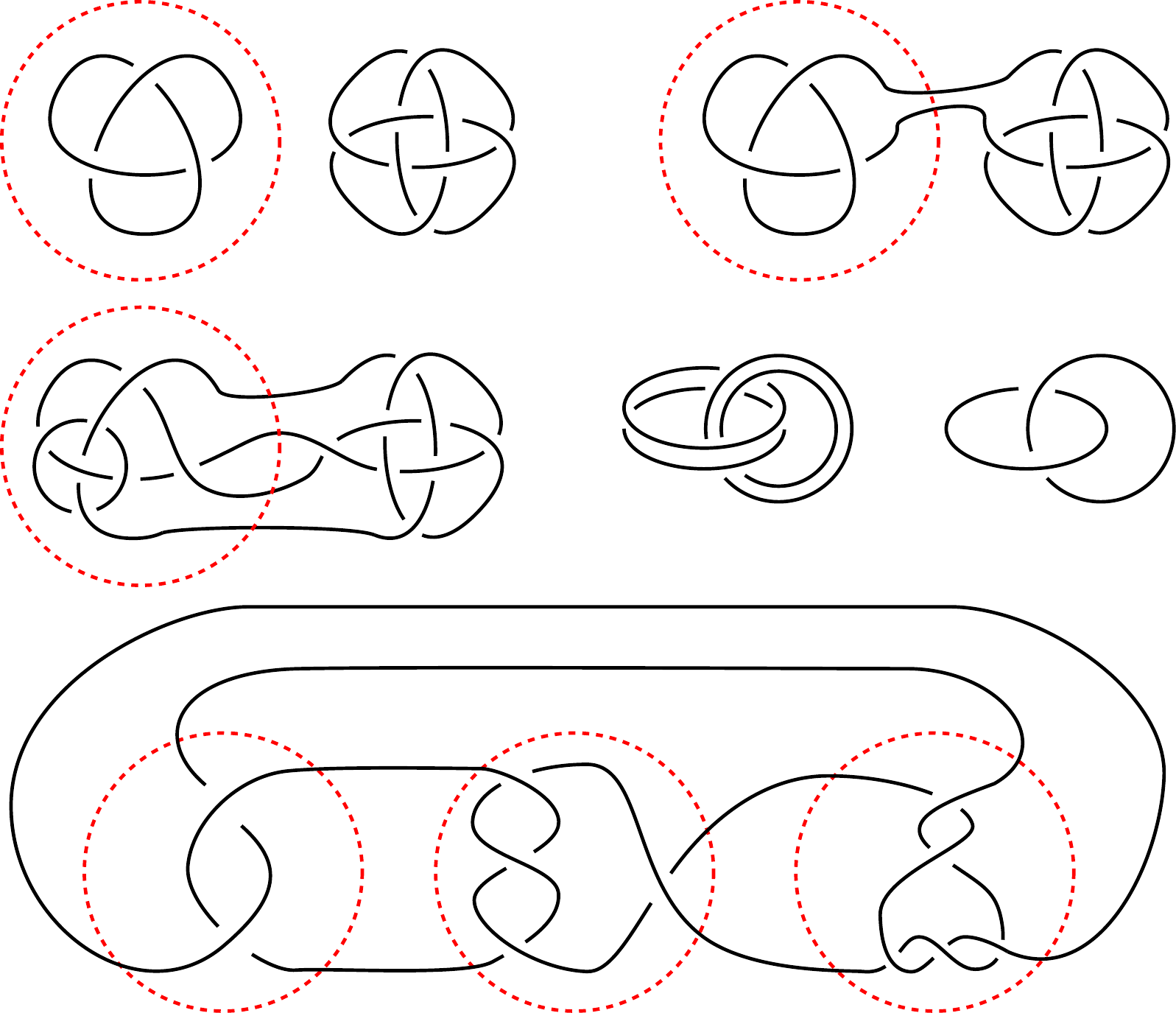}
 \end{center}
 \caption{Obstructions for a double covering to be hyperbolic. Every dashed circle denotes a sphere intersecting the link in 0, 2, or 4 points.}  \label{link_obstructions:fig}
\end{figure}
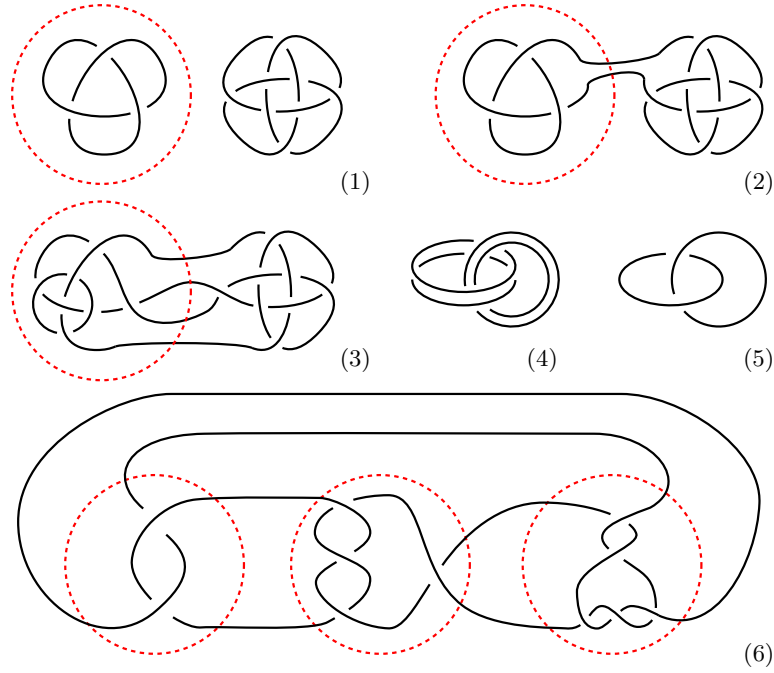

Some examples of the six types of obstructions are shown in Figure \ref{link_obstructions:fig}.
A link $L\subset S^3$ is \emph{split} is there is a sphere $S \subset S^3$ disjoint from $L$ that separates $S^3$ into two discs $D,D'$, each containing some components of $L$. A \emph{Conway sphere} is a sphere $S$ intersecting $L$ transversely in four points, that separates $S^3$ into two discs $D, D'$, such that neither of the four arcs in $D \cap L$ or  $D' \cap L$ can be isotoped to $\partial D$ or $\partial D'$ with endpoints fixed without crossing the rest of $L$ (there is no simple way to reduce the intersection $S\cap L$). Properties (1), (2), (3) can be summarised by saying that there is a sphere intersecting $L$ essentially in at most $4$ points.

A \emph{torus link} is a link entirely contained in a standard unknotted torus in $S^3$. A \emph{Montesinos link} of length $k$ is a more complicated link as in Figure \ref{link_obstructions:fig}, with $k$ spheres in a row ($k=3$ in the figure), that look like Conway spheres but they are actually not, because the arcs contained in the discs can actually be isotoped into these spheres in a non-obvious way (these portions are called \emph{rational tangles}, and each can be described via a rational number $p/q$).

Each of these obstructions (1), \ldots, (6) lifts to an obstruction for the hyperbolicity of  the branched double covering $M$. The preimage in $M$ of the sphere in (1), (2), (3) consists respectively of two essential spheres, one essential sphere, and one essential torus. The preimage in $M$ of the essential torus in (4) consists of either one or two essential tori.
In the remaining cases (5) and (6) the branched covering is well-known and it has some geometry $S^3, \matH^2 \times \matR,$ or $\widetilde {\SL_2}$.

The stated hyperbolization for double branched covers is the result of many contributors, starting from Montesinos  \cite{Monte} and ending with Thurston's geometrisation of orbifolds, whose complete proof is contained in Boileau -- Leeb -- Porti \cite{BLP}. See Bonahon and Siebenmann \cite{BS} for an overview.

\subsection{Dehn surgery}
Let $L\subset S^3$ be a link. A closed regular neighbourhood of $L$ consists of a disjoint union of solid tori, one for each component of $L$ as in Figure \ref{surfaces:fig}. 
A \emph{Dehn surgery} on $L$ is the operation that consists of removing each solid torus, thus creating the exterior $M = S^3 \setminus \nu L$, and gluing it back along a different map. 
The result is a new manifold $N$ without boundary. 

The isotopy class of every regluing map is uniquely determined by a pair $(p,q)$ of coprime integers that encode the isotopy class of the image of the meridian $D \times \{{\rm pt}\}$ of the solid torus $D \times S^1$ in the boundary of the regular neighbourhood of the component of $L$. We have here a combinatorial way of describing many manifolds $N$, out of a link $L$ and some rational numbers $p/q$ attached to its components.
Lickorish \cite{Li} and Wallace \cite{W} proved independently that in fact \emph{every} 3-manifold $N$ without boundary may be constructed in this way. For instance, if we attach the number $0$ to the three components of the Borromean rings in Figure \ref{Borromean:fig}, we get the 3-torus $N= S^1\times S^1 \times S^1$, and if we attach the number 3 to the figure-eight knot, we get the Hantsche -- Wendt flat 3-manifold.

This combinatorial presentation for $N$ can be given to the extremely efficient and easy to use program SnapPy \cite{SnapPy}, that tries to find a hyperbolic structure for $N$, and in the frequent cases where it succeeds, it calculates many interesting geometric invariants of $N$.

In contrast with the previous contexts (knot complements, polyhedra, and double branched coverings), here there is no nice theorem that, given a link $L$ and some rational numbers attached to its components, says whether the resulting manifold $N$ is hyperbolic or not. Thurston's hyperbolic Dehn filling theorem \cite{bibbia} says 

\begin{center}
\emph{If $L$ has hyperbolic exterior and the rational numbers \\ are not contained in some finite set $S \subset \matQ$, then $N$ is hyperbolic.}
\end{center}

The finite set $S$ depends on $L$. For instance, the figure eight knot has hyperbolic exterior, and if we pick any number $p/q \not \in S = \{-4,-3,-2,-1,0,1,2,3,4,\infty\}$, the manifold $N$ is hyperbolic: this famous example was described by Thurston \cite{bibbia}. If we assign a number $p/q \in S$, we get a manifold of some other geometries, or with a non-trivial decomposition.

\begin{figure}
 \begin{center}
\centering
  \includegraphics[width = 4 cm]{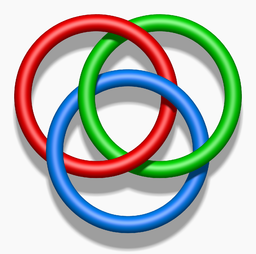}
 \end{center}
 \caption{The Borromean rings. By assigning the label 0 to each component we get the 3-torus.}  \label{Borromean:fig}
\end{figure}

\bibliographystyle{alpha}
\bibliography{biblio}

\end{document}